\pdfoutput=1 
\documentclass[11pt]{amsart}

\usepackage{epsfig,overpic,comment}
\usepackage[usenames,dvipsnames,svgnames,table]{xcolor}
\usepackage[hyphens]{url}
\usepackage[pagebackref,linktocpage=true,colorlinks=true,linkcolor=Blue,citecolor=BrickRed,urlcolor=RoyalBlue]{hyperref}
\usepackage[msc-links,abbrev]{amsrefs}
\usepackage{amsmath,amsthm,amssymb,marginnote}
\usepackage{enumerate}
\usepackage[T1]{fontenc}

\newtheorem{theorem}{Theorem}[section]
\newtheorem{corollary}[theorem]{Corollary}
\newtheorem{lemma}[theorem]{Lemma}
\newtheorem{proposition}[theorem]{Proposition}

\theoremstyle{definition}

\newcommand{\be}{\begin{equation}}
\newcommand{\ee}{\end{equation}}

\newcommand{\R}{\mathbf{R}}

\newcommand{\C}{\mathcal{C}}
\newcommand{\G}{\Gamma}
\newcommand{\g}{\gamma}
\newcommand{\M}{\mathcal{M}}

\renewcommand{\epsilon}{\varepsilon}
\renewcommand{\S}{\mathbf{S}}

\DeclareMathOperator{\cl}{cl}

\makeatletter
\DeclareFontFamily{U}{tipa}{}
\DeclareFontShape{U}{tipa}{m}{n}{<->tipa10}{}
\newcommand{\arc@char}{{\usefont{U}{tipa}{m}{n}\symbol{62}}}%

\newcommand{\arc}[1]{\mathpalette\arc@arc{#1}}

\newcommand{\arc@arc}[2]{%
  \sbox0{$\m@th#1#2$}%
  \vbox{
    \hbox{\resizebox{\wd0}{\height}{\arc@char}}
    \nointerlineskip
    \box0
  }%
}
\makeatother

\makeatletter
\def\@tocline#1#2#3#4#5#6#7{\relax
  \ifnum #1>\c@tocdepth 
  \else
    \par \addpenalty\@secpenalty\addvspace{#2}%
    \begingroup \hyphenpenalty\@M
    \@ifempty{#4}{%
      \@tempdima\csname r@tocindent\number#1\endcsname\relax
    }{%
      \@tempdima#4\relax
    }%
    \parindent\z@ \leftskip#3\relax \advance\leftskip\@tempdima\relax
    \rightskip\@pnumwidth plus4em \parfillskip-\@pnumwidth
    #5\leavevmode\hskip-\@tempdima
      \ifcase #1
       \or\or \hskip 1.3em \or \hskip 2em \else \hskip 5em \fi%
      #6\nobreak\relax
    \hfill\hbox to\@pnumwidth{\@tocpagenum{#7}}\par
    \nobreak
    \endgroup
  \fi}
\makeatother

\newcommand{\nocontentsline}[3]{}
\newcommand{\tocless}[2]{\bgroup\let\addcontentsline=\nocontentsline#1{#2}\egroup}

\begin{document}
\setlength{\baselineskip}{1.2\baselineskip}

\title[Total mean curvatures of Riemannian hypersurfaces] 
{Total mean curvatures of Riemannian hypersurfaces}

\author{Mohammad Ghomi}
\address{School of Mathematics, Georgia Institute of Technology,
Atlanta, GA 30332}
\email{ghomi@math.gatech.edu}
\urladdr{www.math.gatech.edu/~ghomi}

\author{Joel Spruck}
\address{Department of Mathematics, Johns Hopkins University,
 Baltimore, MD 21218}
\email{js@math.jhu.edu}
\urladdr{www.math.jhu.edu/~js}

\vspace*{-0.75in}
\begin{abstract}
We obtain a comparison formula for integrals of mean curvatures of Riemannian hypersurfaces via Reilly's identities. As applications we derive several geometric inequalities  for a convex hypersurface $\G$ in a Cartan-Hadamard manifold $M$. In particular we show that the first mean curvature integral of a convex hypersurface $\gamma$ nested inside $\G$ cannot exceed that of $\G$, 
which leads in dimension $3$ to a sharp lower bound for the total first mean curvature of $\Gamma$ in terms of the volume it bounds in $M$.  This monotonicity property is extended to all mean curvature integrals when $\g$ is parallel to $\G$, or $M$ has constant curvature. We also 
characterize hyperbolic balls as minimizers of  the mean curvature integrals among balls with equal radii in Cartan-Hadamard manifolds.
\end{abstract}

\date{\today \,(Last Typeset)}
\thanks{The research of M.G. was supported by NSF grant DMS-2202337 and a Simons Fellowship. The research of J.S. was supported by a Simons Collaboration Grant}

\maketitle


\section{Introduction}\label{sec:intro}
Total mean curvatures of a hypersurface $\G$ in a Riemannian $n$-manifold $M$ are integrals of symmetric functions of its principal curvatures. These quantities are known as quermassintegrals or mixed volumes when $\G$ is convex and $M$ is the Euclidean space. They
are fundamental  in geometric variational problems, as they feature in Steiner's polynomial, Brunn-Minkowski theory, and Alexandrov-Fenchel inequalities \cite{schneider2014,santalo:1976,gardner2002,gray2004}, which were all originally developed in Euclidean space. Extending these notions to Riemannian manifolds has been a major topic of investigation. In particular, total mean curvatures have been studied extensively in hyperbolic space in recent years \cite{andrews-hu-li-2020, solanes2006,wang-xia2014,wei-xiong2015}. Here we study these integrals in the broader setting of  \emph{Cartan-Hadamard spaces}, i.e., complete simply connected manifolds of nonpositive curvature, and generalize a number of inequalities which had been established  in Euclidean or hyperbolic space.


The main result of this paper, Theorem \ref{thm:comparison}, expresses the difference between the total $r^{th}$ mean curvatures of a pair of nested hypersurfaces $\G$, $\g$ in a Riemannian manifold $M$ in terms of sectional curvatures of $M$  and principal curvatures of a family of  hypersurfaces which fibrate the region between $\G$ and $\g$. This formula simplifies when $r=1$,  $\G$ and $\g$ are parallel, or  $M$ has constant curvature, leading to a number of applications. In particular we establish the monotonicity property of the total first mean curvature for nested convex hyeprsurfaces in  Cartan-Hadamard manifolds (Corollary \ref{cor:sigma1}). This  leads to a sharp lower bound in dimension $3$ for the total first mean curvature in terms of the volume bounded by $\Gamma$ (Corollary \ref{cor:M1vol}), which generalizes a   result of Gallego-Solanes in hyperbolic $3$-space \cite[Cor. 3.2]{gallego-solanes2005}. We also extend to 
all mean curvatures some monotonicity results of Schroeder-Strake \cite{schroeder-strake} and Borbely \cite{borbely2002} for total Gauss-Kronecker curvature (Corollaries \ref{cor:parallel} and \ref{cor:cc}).  Finally we include a characterization of hyperbolic balls as minimizers of total mean curvatures among balls of equal radii in Cartan-Hadamard manifolds (Corollary \ref{cor:balls}). 

Theorem \ref{thm:comparison} is a generalization of the comparison result we had obtained earlier in  \cite{ghomi-spruck2022} for the  Gauss-Kronecker curvature,  motivated by Kleiner's approach to the Cartan-Hadamard conjecture on the isoperimetric inequality \cite{kleiner1992}. 
Similar to \cite{ghomi-spruck2022}, our starting point here, in Section \ref{sec:Reilly}, will be an identity  (Lemma \ref{lem:reilly1}) for divergence of Newton operators, which were developed by Reilly \cites{reilly1973, reilly1977} to study invariants of Hessians of functions on Riemannian manifolds. This formula together with Stokes' theorem leads to the proof of Theorem \ref{thm:comparison} in Section \ref{sec:comparison}. Then, in Section \ref{sec:applications}, we will develop the applications of that result.

\section{Newton Operators}\label{sec:Reilly}

Throughout this work $M$ denotes an $n$-dimensional Riemannian manifold with metric $\langle \cdot, \cdot \rangle$ and covariant derivative $\nabla$. Furthermore, $u$ is a $\C^{1,1}$ function on $M$. In particular, $u$ is twice differentiable at almost every point $p$ of $M$, and the computations below take place at such a point. 
The \emph{gradient} of $u$ is the tangent vector $\nabla u\in T_p M$  given by
$
\langle \nabla u(p), X\rangle :=\nabla_X u,
$
for all $X\in T_p M$.
The \emph{Hessian operator} $\nabla^2 u\colon T_p M\to T_pM$  is the self-adjoint linear map  given by
$
\nabla^2 u(X):=\nabla_X(\nabla u).
$
The \emph{symmetric elementary functions} $\sigma_r\colon \R^k\to \R$, for $1\leq r\leq k$, and $x=(x_1,\dots,x_k)$ are defined by 
$$
\sigma_r(x):=\sum_{i_1<\dots<i_r}x_{i_1}\dots x_{i_r}.
$$
We  set $\sigma_0:=1$, and $\sigma_r:=0$ for $r\geq k+1$ by convention. Let $\lambda(\nabla^2u):=(\lambda_1,\dots,\lambda_n)$ denote the eigenvalues of $\nabla^2u$. Then we set
$$
\sigma_r(\nabla^2u):=\sigma_r\big(\lambda(\nabla^2u)\big).
$$ 
These functions form the coefficients of the characteristic polynomial 
$$
P(\lambda):=\det(\lambda\, I-\nabla^2u)=\sum_{i=0}^n(-1)^i\sigma_i(\nabla^2 u)\lambda^{n-i}.
$$
Let $\delta^{i_1\ldots i_m}_{j_1 \ldots j_m}$ be the generalized \emph{Kronecker  tensor}, which is equal to $1$ ($-1$) if 
$i_1,\ldots,i_m$ are distinct and $(j_1,\ldots,j_m)$ is an even (odd) permutation of $(i_1,\ldots,i_m)$; otherwise, it is equal to $0$. Then \cite[Prop. 1.2(a)]{reilly1973}, 
\be\label{eq:sigma_r}
\sigma_r(\nabla^2u)=\frac1{r!}\delta^{i_1\ldots i_r}_{j_1\ldots j_r}u_{i_1 j_1}\cdots u_{i_r j_r},
\ee
where $u_{ij}:=\nabla_{ij}u$ denote the second partial derivatives of $u$ with respect to an orthonormal frame $E_i\in T_p M$, which we extend to an open neighborhood of $p$ by parallel translation along geodesics. So $\nabla_{E_i}E_j=0$ at $p$. We call $E_i$ a \emph{local parallel frame} centered at $p$, and set $\nabla_i:=\nabla_{E_i}$,  $\nabla_{ij}:=\nabla_i\nabla_j$.  Each of the indices in \eqref{eq:sigma_r} ranges from $1$ to $n$, and we employ \emph{Einstein's convention}  by summing over repeated indices throughout the paper. The \emph{Newton operators} $\mathcal{T}_r^{\,u}\colon T_p M\to T_p M$ \cites{reilly1973, reilly1977}  are defined recursively by setting 
$\mathcal{T}_0^{\,u}:=I$, the identity map, and for $r\geq 1$,
\begin{gather} \label{newton}
\mathcal{T}_r^{\,u}:=\sigma_r(\nabla^2u)\,  I-\mathcal{T}_{r-1}^{\,u}\circ \nabla^2 u=\sum_{i=0}^r(-1)^i\sigma_i(\nabla^2 u)(\nabla^2u)^{r-i}.
\end{gather}
Thus $\mathcal{T}_r^{\,u}$ is the truncation of the polynomial $P(\nabla^2u)$ obtained by removing the terms of order higher than $r$.
In particular $\mathcal{T}_n^{\,u}=P(\nabla^2u)$. So, by the Cayley-Hamilton theorem, $\mathcal{T}_n^{\,u}=0$. Consequently, when $\nabla^2 u$ is nondegenerate, \eqref{newton} yields that 
\be\label{eq:Tn-1=T}
\mathcal{T}_{n-1}^{\,u}=\sigma_n(\nabla^2u)(\nabla^2u)^{-1}=\det(\nabla^2u)(\nabla^2u)^{-1}=\mathcal{T}^{\,u},
\ee
where $\mathcal{T}^{\,u}$ is the Hessian cofactor operator discussed in \cite[Sec. 4]{ghomi-spruck2022}.
See \cite[Prop. 1.2]{reilly1973} for other basic identities which relate $\sigma$ and $\mathcal{T}$. In particular,  by \cite[Prop. 1.2(c)]{reilly1973}, we have $\text{Trace}(\mathcal{T}_r^{\,u}\cdot\nabla^2u)=(r+1)\sigma_{r+1}(\nabla^2u)$. So, by Euler's identity for homogeneous polynomials,
\be\label{eq:Tr^u_ij}
(\mathcal{T}_r^{\,u})_{ij}u_{ij}=\text{Trace}(\mathcal{T}_r^{\,u}\circ\nabla^2u)=(r+1)\sigma_{r+1}(\nabla^2u)=\frac{\partial \sigma_{r+1}(\nabla^2u)}{\partial u_{ij}}u_{ij}.
\ee
Thus it follows from \eqref{eq:sigma_r} that 
\be\label{eq3.10A}
\left(\mathcal{T}_r^{\,u}\right)_{ij}=\frac{\partial \sigma_{r+1}(\nabla^2u)}{\partial u_{ij}}\,=\,\frac1{r!}\delta^{i i_1\ldots i_r}_{j j_1\ldots j_r}u_{i_1 j_1}\cdots u_{i_r j_r}.
\ee
 Furthermore, by \cite[Prop. 1(11)]{reilly1977} (note that the sign of the Riemann tensor $R$ in \cite{reilly1977}  is opposite to the one in this paper) we have
\be \label{eq3.10B}
\big(\text{div}(\mathcal{T}_r^{\,u})\big)_j=\frac1{(r-1)!}\delta^{i i_1 \ldots i_{r}} _{j  j_1 \ldots j_r }u_{i_1 j_1}\cdots u_{i_{r-1} j_{r-1}} R_{i j_r i_r k}u_k,
\ee
where $R_{ijkl}:=R(E_i, E_j, E_k, E_\ell)=\langle\nabla_i\nabla_jE_k-\nabla_j\nabla_iE_k,E_\ell\rangle$.
Another useful identity \cite[p. 462]{reilly1977} is
\be\label{eq:divTr}
\text{div}\big(\mathcal{T}_r^{\,u}(\nabla u)\big)= 
\big\langle \mathcal{T}_r^{\,u},\nabla^2 u\big\rangle+\big\langle\text{div}(\mathcal{T}_r^{\,u}),  \nabla u\big\rangle,
\ee
where $\langle\cdot,\cdot\rangle$ here indicates the Frobenius inner product (i.e., $\langle A, B\rangle:=A_{ij}B_{ij}$ for any pair of matrices of the same dimension). Divergence of  $\mathcal{T}^{\,u}_r$ may be defined by virtually the same argument used for $\mathcal{T}^{\,u}$ in  \cite[Sec. 4]{ghomi-spruck2022} to yield the following generalization of   \cite[(14)]{ghomi-spruck2022}:
\be\label{eq:divTk}
\big(\text{div}(\mathcal{T}^{\,u}_r)\big)_j=\nabla_i(\mathcal{T}^{\,u}_r)_{ij}.
\ee
Recall that $\mathcal{T}^{\,u}=\mathcal{T}^{\,u}_{n-1}$ by \eqref{eq:Tn-1=T}. Furthermore,
$\mathcal{T}^{\,u}_n=0$ as we mentioned earlier. Thus the following observation generalizes  \cite[Lem. 4.2]{ghomi-spruck2022}.
\begin{lemma}\label{lem:reilly1}
$$
\textup{div}\left(\mathcal{T}^{\,u}_{r-1}\left(\frac{\nabla u}{|\nabla u|^r}\right)\right)
=
\left\langle\textup{div}(\mathcal{T}^{\,u}_{r-1}), \frac {\nabla u}{|\nabla u|^r}\right\rangle
+
r\frac{\big\langle \mathcal{T}^{\,u}_{r}(\nabla u), \nabla u\big\rangle}{|\nabla u|^{r+2}}.
$$
\end{lemma}
\begin{proof} 
By Leibniz rule and \eqref{eq:divTk} we have
\begin{eqnarray*}
\textup{div}\left(\mathcal{T}^{\,u}_{r-1}\left(\frac{\nabla u}{|\nabla u|^r}\right)\right)&=& \nabla_i\left((\mathcal{T}^{\,u}_{r-1})_{ij} \frac{u_j}{|\nabla u|^r}\right) \\
&=&
\left\langle\text{div}(\mathcal{T}^{\,u}_{r-1}), \frac {\nabla u}{|\nabla u|^r}\right\rangle+
(\mathcal{T}^{\,u}_{r-1})_{ij}\left(\frac{u_{ij}}{|\nabla u|^r}-r \frac{u_j u_\ell u_{\ell i}}{|\nabla u|^{r+2}}\right),
\end{eqnarray*}
where the computation to obtain the second term on the right is identical to the one performed earlier in   \cite[Lem. 4.2]{ghomi-spruck2022}. To develop this term further, note that by 
\eqref{newton}
$$
(\mathcal{T}^{\,u}_{r-1})_{ij} u_{\ell i}=\sigma_r(\nabla^2u) \delta_{\ell j}-(\mathcal{T}^{\,u}_r)_{\ell j},
$$
which in turn yields
$$
(\mathcal{T}^{\,u}_{r-1})_{ij} u_{\ell i}\frac{u_j u_\ell}{|\nabla u|^2}=
\sigma_r(\nabla^2u)-(\mathcal{T}^{\,u}_{r})_{ij}\frac{u_i u_j}{|\nabla u|^2}.
$$
Hence 
\begin{eqnarray*}
(\mathcal{T}^{\,u}_{r-1})_{ij}\left(\frac{u_{ij}}{|\nabla u|^r}-r \frac{u_j u_\ell u_{\ell i}}{|\nabla u|^{r+2}}\right)&=&\frac{r\sigma_r(\nabla^2u)}{|\nabla u|^r}-\frac{r}{|\nabla u|^r}\left(\sigma_r(\nabla^2u)-(\mathcal{T}^{\,u}_{r})_{ij}\frac{u_i u_j}{|\nabla u|^2}\right)\\
&=&r(\mathcal{T}^{\,u}_r)_{ij}\frac{u_i u_j}{|\nabla u|^{r+2}},
\end{eqnarray*}
which completes the proof.
\end{proof}

Below we assume, as was the case in \cite[Sec. 4]{ghomi-spruck2022}, that all local computations take place with respect to a \emph{principal curvature frame} $E_i\in T_p M$ of $u$, which is defined as follows. Assuming $|\nabla u(p) |\neq 0$,  we set $E_n:=\nabla u(p)/|\nabla u(p) |$, and let $E_1, \dots, E_{n-1}$ be the principal directions of the level set of $u$ passing through $p$. Then we extend $E_i$ to a local parallel frame near $p$. The first partial derivatives of $u$ with respect to $E_i$, $u_i:=\nabla_i u$, satisfy
\be\label{eq:ui}
u_i=0 \;\;\text{for}\;\; i\neq n, 
\quad\quad\text{and}\quad\quad
u_n=|\nabla u|.  
\ee
Furthermore, for the second partial derivatives $u_{ij}=\nabla_{ij}u$ we have
\be\label{eq:uii}
u_{ij}=0, \;\;\text{for}\;\; i\neq j\leq n-1,
\quad\quad\text{and}\quad\quad
\frac{u_{ii}}{|\nabla u|}=:\kappa_i^u,\; \;\;\text{for}\;\; i\neq n, 
\ee
where $\kappa_1^u,\dots,\kappa_{n-1}^u$ are the \emph{principal curvatures} of level sets of $u$ with respect to $E_n$, i.e., they are eigenvalues corresponding to $E_1,\dots, E_{n-1}$ of the shape operator $X\mapsto \nabla_X\nu$ on the tangent space of level sets of $u$, where $\nu:=\nabla u/|\nabla u|$.
We set $\kappa_u:=(\kappa_1^u,\dots, \kappa_{n-1}^u)$. So
$
\sigma_r(\kappa^u)
$
is the $r^{th}$  \emph{mean curvature} of the level set of $u$ at $p$. In particular, $\sigma_{n-1}(\kappa^u)$ is the \emph{Gauss-Kronecker curvature} of the level sets. The next observation generalizes \cite[Lem. 4.1]{ghomi-spruck2022}.

\begin{lemma}\label{lem:reilly2} 
$$
\sigma_r(\kappa^u)=\frac{\big\langle \mathcal{T}^{\,u}_{r}(\nabla u), \nabla u\big\rangle}{|\nabla u|^{r+2}}.
 $$
\end{lemma}
\begin{proof} 
 \eqref{eq3.10A} together with \eqref{eq:ui} and \eqref{eq:uii} yields that
\begin{eqnarray*}
(\mathcal{T}^{\,u}_{r})_{ij}\frac{u_i u_j}{|\nabla u|^{r+2}}&=&\frac1{r!}\delta^{ii_1\cdots i_r}_{j j_1 \cdots j_r}u_{i_1 j_1} \cdots u_{i_r j_r} \frac{u_i u_j}{|\nabla u|^{r+2}}\\
&=&\frac1{r!}\delta^{ni_1\cdots i_r}_{nj_1 \cdots j_r}
\frac{u_{i_1 j_1} \cdots u_{i_r j_r} }{|\nabla u|^r}\cdot\frac{u_n^2}{|\nabla u|^2}\\
&=&\frac1{r!}\delta^{ni_1\cdots i_r}_{ni_1 \cdots i_r}
\kappa^u_{i_1}\dots\kappa^u_{i_r}.
\end{eqnarray*}
\end{proof}

\section{The Comparison Formula}\label{sec:comparison}
Here we establish the main result of this work.  For a $\C^{1,1}$ hypersurface $\G$ in a Riemannian $n$-manifold $M$, oriented by a choice of normal vector field $\nu$, and $0\leq r\leq n-1$, we let 
$$
\M_{r}(\G):=\int_\G \sigma_r(\kappa)\, 
$$
be the  \emph{total $r^{th}$  mean curvature} of $\Gamma$, where 
$\kappa:=(\kappa_1,\dots,\kappa_{n-1})$ denotes  principal curvatures of $\G$ with respect to $\nu$.  
Note that $\M_{0}(\G)=|\G|$, the volume of $\G$, since $\sigma_0=1$, and 
$\M_{n-1}(\G)$  is the \emph{total Gauss-Kronecker curvature} of $\G$ (denoted by $\mathcal{G}(\G)$ in \cite{ghomi-spruck2022}). A \emph{domain} $\Omega\subset M$ is an open set with compact closure $\cl(\Omega)$. 
If $\Gamma$ bounds a domain $\Omega$, then by convention we  set $\M_{-1}(\G):=|\Omega|$, the volume of  $\Omega$.
The following theorem generalizes \cite[Thm. 4.7]{ghomi-spruck2022} where this result had been established for $r=n-1$. It also uses less regularity than was required in \cite[Thm. 4.7]{ghomi-spruck2022}.

\begin{theorem}\label{thm:comparison}
Let $\G$ and $\g$ be closed $\C^{1,1}$ hypersurfaces in a Riemannian $n$-manifold $M$ bounding domains $\Omega$ and $D$ respectively, with $\cl(D)\subset\Omega$. Suppose there exists a $\C^{1,1}$ function $u$ on $\cl(\Omega\setminus D)$ with $\nabla u\neq 0$ which is constant on $\G$ and $\g$. Let $\kappa^u:=(\kappa_1^u,\dots,\kappa_{n-1}^u)$ be principal curvatures of level sets of $u$ with respect to $E_n:=\nabla u/|\nabla u|$, and let  $E_1,\dots, E_{n-1}$ be the corresponding principal directions. Then, 
for $0\leq r\leq n-1$,
\begin{multline*}
\M_{r}(\G)-\M_{r}(\g)
=
(r+1)\int_{\Omega\setminus D}\sigma_{r+1}(\kappa^u)\, \\
+
\int_{\Omega\setminus D}\left(-\sum\kappa^u_{i_1}\dots\kappa^u_{i_{r-1}}K_{i_{r}n}
+
\frac{1}{|\nabla u|}\sum\kappa^u_{i_1}\dots\kappa^u_{i_{r-2}}|\nabla u|_{i_{r-1}}R_{i_{r}i_{r-1}i_{r}n}\right) ,
\end{multline*}
where $|\nabla u|_i:=\nabla_{E_i} |\nabla u|$, $R_{ijkl}=R(E_i, E_j, E_k, E_l)$ are components of the Riemann curvature tensor of $M$, $K_{ij}=R_{ijij}$ is the sectional curvature, and
the summations take place over distinct values of $1\leq i_1,\dots, i_r\leq n-1$,  
with $i_1<\dots <i_{r-1}$ in the first sum, and $i_1<\dots <i_{r-2}$ in the second sum. 
\end{theorem}
\begin{proof}
By Lemmas \ref{lem:reilly1} and \ref{lem:reilly2},
\be\label{eq:comparison1}
\textup{div}\left(\mathcal{T}^{\,u}_{r}\left(\frac{\nabla u}{|\nabla u|^{r+1}}\right)\right)
=
(r+1) \sigma_{r+1}(\kappa^u)
+
\left\langle\textup{div}(\mathcal{T}^{\,u}_{r}), \frac {\nabla u}{|\nabla u|^{r+1}}\right\rangle.
\ee
By Stokes' theorem and Lemma \ref{lem:reilly2},
$$
\int_{\Omega\setminus D}\textup{div}\left(\mathcal{T}^{\,u}_{r}\left(\frac{\nabla u}{|\nabla u|^{r+1}}\right)\right) 
=
\int_{\G\cup\g}\left\langle\mathcal{T}^{\,u}_{r}\left(\frac{\nabla u}{|\nabla u|^{r+1}}\right),\frac{\nabla u}{|\nabla u|}\right\rangle  
=
\M_{r}(\G)-\M_{r}(\g).
$$
So integrating both sides of \eqref{eq:comparison1} yields
$$
\M_{r}(\G)-\M_{r}(\g)=(r+1)\int_{\Omega\setminus D}\sigma_{r+1}(\kappa^u)\, 
+
\int_{\Omega\setminus D} \left\langle\textup{div}(\mathcal{T}^{\,u}_{r}), \frac {\nabla u}{|\nabla u|^{r+1}}\right\rangle.
$$
Using \eqref{eq3.10B} and \eqref{eq:ui}, we have
\begin{eqnarray*}
\left\langle\textup{div}(\mathcal{T}^{\,u}_{r}), \frac {\nabla u}{|\nabla u|^{r+1}}\right\rangle
&=&
\frac1{(r-1)!}\delta^{i\, i_1 \ldots i_{r}} _{j  j_1\ldots j_{r} }u_{i_1 j_1}\cdots u_{i_{r-1} j_{r-1}} R_{i j_{r} i_{r} k}\frac{u_ku_j}{|\nabla u|^{r+1}}\\
&=&
\frac{1}{(r-1)!}\delta^{i\, i_1 \ldots i_{r}} _{n  j_1\ldots j_{r} }\frac{u_{i_1 j_1}}{|\nabla u|}\cdots \frac{u_{i_{r-1} j_{r-1}}}{|\nabla u|} R_{ij_{r} i_{r} n}.
\end{eqnarray*}
The last expression may be written as  the sum of two components $A$ and $B$  which consist of terms with $i=n$ and $i\neq n$ respectively. Note that we may assume $j_1,\dots, j_r\neq n$,  for otherwise $\delta^{i\,i_1 \ldots i_{r}} _{n  j_1\ldots j_{r} }=0$. 
To compute $A$ note that if $i=n$, then for $\delta^{i\,i_1 \ldots i_{r}} _{n  j_1\ldots j_{r} }$ not to vanish, we must have $i_1,\dots, i_r\neq n$. Then by \eqref{eq:uii} $u_{i_kj_k}=0$ unless $i_k=j_k$, which yields that 
$$ 
A=\frac{1}{(r-1)!}\delta^{n i_1 \ldots i_{r}} _{n  i_1\ldots i_{r} }\frac{u_{i_1 i_1}}{|\nabla u|}\cdots \frac{u_{i_{r-1} i_{r-1}}}{|\nabla u|} R_{ni_{r} i_{r} n}=-\sum\kappa^u_{i_1}\dots\kappa^u_{i_{r-1}}K_{i_{r}n},
$$
where the sum ranges over all distinct values of  $1\leq i_1,\dots, i_r\leq n-1$, with $i_1<\dots< i_{r-1}$
as desired. To find $B$ note that if $i\neq n$, then for $\delta^{i\,i_1 \ldots i_{r}} _{n  j_1 \ldots j_{r} }$ not to vanish, we must have $i_k=n$ for some $1\leq k\leq r$. If $k=r$, then $R_{ij_{r}i_{r}n}=R_{ij_{k}nn}=0$. In particular, $B=0$ when $r=1$. Now assume that $r\geq 2$. Then we may assume that $k\neq r$, or $i_r\neq n$.  Then, by \eqref{eq:uii}, $u_{i_r j_r}=0$ unless $i_r=j_r$. So we may assume that $i_r=j_r$  for $r\neq k$, which  in turn implies that $j_k=i$. Thus $B=\sum_{k=1}^{r-1} B_k$ where
\begin{gather*}
B_k=\frac{1}{(r-1)!}\delta^{i\,i_1 \ldots i_{k-1}n i_{k+1}\ldots i_{r}} _{n  i_1\ldots i_{k-1}\,i\,i_{k+1}\ldots i_{r} }\frac{u_{i_1i_1}}{|\nabla u|}\dots \frac{u_{i_{k-1}i_{k-1}}}{|\nabla u|}\frac{u_{n i}}{|\nabla u|}\frac{u_{i_{k+1}i_{k+1}}}{|\nabla u|}\dots \frac{u_{i_{r-1}i_{r-1}}}{|\nabla u|} R_{ii_{r}  i_{r} n}\\
=
\frac{-1}{(r-1)}\sum\kappa^u_{i_1}\dots \kappa^u_{i_{k-1}} \frac{|\nabla u|_i}{|\nabla u|}\kappa^u_{i_{k+1}}\dots \kappa^u_{i_{r-1}}R_{ii_{r} i_{r} n},
\end{gather*}
since $u_n=|\nabla u|$.
Here the sum ranges over all distinct indices $1\leq i, i_1,\dots, i_{k-1},i_k,\dots, i_r\leq n-1$, with $i_1<\dots<i_{k-1}<i_{k+1}<\dots<i_{r-1}$.  Note that $B_1=\dots=B_{r-1}$.
Thus
$$
B=(r-1)B_{r-1}=\frac{1}{|\nabla u|}\sum\kappa^u_{i_1}\dots  \kappa^u_{i_{r-2}}|\nabla u|_iR_{i_{r}i i_{r} n},
$$
which completes the proof (after renaming  $i$ to $i_{r-1}$).
\end{proof}

\section{Applications}\label{sec:applications}
Here we develop some consequences of Theorem \ref{thm:comparison}.  A subset of a Cartan-Hadamard manifold $M$ is \emph{convex} if it contains the (unique) geodesic segment connecting every pair of its points.  A \emph{convex hypersurface} $\G\subset M$ is the boundary of a compact convex set with interior points. If $\G$ is of class $\C^{1,1}$, then its principal curvatures are nonnegative at all twice differentiable points, with respect to the outward normal. Conversely, if the principal curvatures of a closed hypersurface $\G\subset M$ are all nonnegative, then $\G$ is convex \cite{alexander1977}. See \cite[Sec. 2 and 3]{ghomi-spruck2022} for basic properties of convex sets  in Cartan-Hadamard manifolds. A set  is \emph{nested inside}  $\G$ if it lies in the convex domain bounded by $\G$.

\begin{corollary}\label{cor:sigma1}
Let $\G$, $\g$ be $\C^{1,1}$ convex hypersurfaces in a Cartan-Hadamard $n$-manifold. Suppose that $\g$ is nested inside $\G$. Then $\M_1(\G)\geq\M_1(\g)$.
\end{corollary}
\begin{proof}
 Setting $r=1$ in the comparison formula of Theorem \ref{thm:comparison} yields
\be\label{eq:G1}
\M_{1}(\G)-\M_{1}(\g)
=
2\int_{\Omega\setminus D}\sigma_{2}(\kappa^u)\, 
-
\int_{\Omega\setminus D}\textup{Ric}\left(\frac{\nabla u}{|\nabla u|}\right),
\ee
where Ric stands for Ricci curvature; more explicitly, in a principal curvature frame where $E_n:=\nabla u/|\nabla u|$, $\textup{Ric}(E_n)$ is the sum of sectional curvatures $K_{in}$, for $1\leq i\leq n-1$. So $\textup{Ric}(E_n)\leq 0$. If $\G$, $\g$ are smooth ($\C^\infty$) and strictly convex, we may let $u$ in Theorem \ref{thm:comparison} be a function with convex level sets \cite[Lem. 1]{borbely2002}. Then $\sigma_{2}(\kappa^u)\geq 0$, which yields $\M_1(\G)\geq\M_1(\g)$ as desired. This completes the proof since  we may approximate $\Gamma$ and $\gamma$ by smooth strictly convex hypersurfaces, e.g., by applying the Greene-Wu convolution to their distance functions, see \cite[Lem. 3.3]{ghomi-spruck:minkowski}; furthermore, total mean curvatures will converge here since they constitute ``valuations'' in the sense of integral geometry, see \cite[Note 3.7]{ghomi-spruck2022} or \cite[Prop. 3.8]{bering-brocker}. 
\end{proof}

Dekster \cite{dekster1981} constructed examples of nested convex hypersurfaces  in Cartan-Hadamard manifolds where the monotonicity property in the last result does not hold for Gauss-Kronecker curvature.  So the above corollary cannot be extended to all mean curvatures without further assumptions, which we will discuss below. First we need to record the following observation.

\begin{lemma}\label{lem:Mr0}
Let $S_\rho$ be a geodesic sphere of radius $\rho$ centered at a point  in a Riemannian manifold. As $\rho\to 0$, $\M_r(S_\rho)$ converges to $0$   for $r\leq n-2$, and to $|\S^{n-1}|$ for $r=n-1$.
\end{lemma}
\begin{proof}
A power series expansion \cite[Thm. 3.1]{chen-vanhecke1981} of the second fundamental form of $S_\rho$ in normal coordinates  shows that  the principal curvatures of $S_\rho$ are given by $\kappa_i^\rho=(1+O(\rho^2))/\rho$. So 
$$
\sigma_r(\kappa^\rho) = \left(\begin{matrix}n-1\\r\end{matrix}\right)\frac1{\rho^r}(1+O(\rho^2)).
$$
Another power series expansion \cite[Thm. 3.1]{gray1974} yields 
$$
|S_\rho|= |\S^{n-1}|\rho^{n-1}(1+O(\rho^2)).
$$
So it follows that
$$
\mathcal{M}_r(S_\rho)= \left(\begin{matrix}n-1\\r\end{matrix}\right)|\S^{n-1}|\rho^{n-1-r}\big(1+O(\rho^2)\big),
$$
which completes the proof.
\end{proof}

Gallego and Solanes showed \cite[Cor. 3.2]{gallego-solanes2005} that if $\Gamma$ is a convex hypersurface bounding a domain $\Omega$ in a hyperbolic $n$-space of constant curvature $a< 0$, then
$$
\M_1(\G)> -(n-1)^2a|\Omega|.
$$
When comparing formulas, note that in \cite{gallego-solanes2005} mean curvature is defined as the \emph{average} of $\kappa_i$, as opposed to the sum of $\kappa_i$, which is our convention. Large balls show that the above inequality  is sharp. Here we extend this inequality to Cartan-Hadamard $3$-manifolds:

\begin{corollary}\label{cor:M1vol}
Let $\G$ be a $\C^{1,1}$ convex hypersurface in a Cartan-Hadamard $n$-manifold $M$, bounding a domain $\Omega$. Suppose that curvature of $M$ is bounded above by $a\leq 0$. Then 
$$
\M_1(\G)> -(n-1)a|\Omega|.
$$
Furthermore, if $n=3$, then
$$
\M_1(\G)> -4a|\Omega|.
$$
\end{corollary}
\begin{proof}
Let $\gamma=\gamma_\rho$ in \eqref{eq:G1} be a  geodesic sphere of radius $\rho$. By Lemma \ref{lem:Mr0}, $\mathcal{M}_1(\gamma_\rho)\to 0$ as $\rho\to0$, which yields
$$
\M_{1}(\G)
=
2\int_{\Omega}\sigma_{2}(\kappa^u)\, 
-
\int_{\Omega}\textup{Ric}\left(\frac{\nabla u}{|\nabla u|}\right)>-(n-1)a|\Omega|,
$$
as desired. When $n=3$, Gauss' equation states that
$$
\sigma_{2}(\kappa^u)=K^u-K_M^u,
$$
where $K^u$ is the sectional curvature of level sets of $u$ and $K_M^u$ is the sectional curvature of $M$ with respect to tangent planes to level sets of $u$. Thus, 
$$
\M_{1}(\G)
=
2\int_{\Omega}K^u\, 
-
2\int_{\Omega}K_M^u
-
\int_{\Omega}\textup{Ric}\left(\frac{\nabla u}{|\nabla u|}\right)>-4a|\Omega|,
$$
which completes the proof.
\end{proof}

  We say $\G$ is an \emph{outer parallel hypersurface} of a convex hypersurface $\g$ if all points of $\Gamma$ are at a  constant  distance $\lambda\geq 0$ from the convex domain bounded by $\g$. Since the distance function of a convex set in a Cartan-Hadamard manifold is convex \cite[Prop. 2.4]{bridson-haefliger1999},  $\G$ is convex. Furthermore $\G$ is $\C^{1,1}$ for $\lambda>0$ \cite[Lem. 2.6]{ghomi-spruck2022}. The following corollary generalizes \cite[Cor. 5.3]{ghomi-spruck2022} and a theorem of Schroeder-Strake \cite[Thm. 3]{schroeder-strake} where this result had been established for Gauss-Kronecker curvature; see also \cite[Note 6.9]{ghomi-spruck2022}.

\begin{corollary} \label{cor:parallel} 
Let $M$ be a Cartan-Hadamard $n$-manifold, and $\G$, $\g$ be $\C^{1,1}$ convex hypersurfaces in $M$. Suppose that $\G$ is an outer parallel hypersurface of $\g$. Then $\M_r(\G)\geq\M_r(\g)$, for $1\leq r\leq n-1$.
\end{corollary}
\begin{proof}
We may let $u$ in Theorem \ref{thm:comparison} be the distance function of the convex domain bounded by $\G$. Then $|\nabla u|$ is constant on level sets of $u$. So $|\nabla u|_i=0$ for $1\leq i\leq n-1$, which yields
$$
\M_r(\G)-\M_r(\g)\geq(r+1)\int_{\Omega\setminus D}\sigma_{r+1}(\kappa^u)\, 
-a(n-r)
\int_{\Omega\setminus D}\sigma_{r-1}(\kappa^u),
$$ 
where $a\leq 0$ is the upper bound for sectional curvatures of $M$. Since $u$ is convex, $\sigma_r(\kappa^u)\geq 0$, which completes the proof.
\end{proof}

The next result generalizes \cite[Cor. 5.2]{ghomi-spruck2022} and an observation of  Borbely \cite[Thm. 1]{borbely2002}  for Gauss-Kronecker curvature. 

\begin{corollary} \label{cor:cc} 
Let $M$ be a Cartan-Hadamard $n$-manifold with constant curvature, and $\G$, $\g$ be $\C^{1,1}$ convex hypersurfaces in $M$, with $\g$ nested inside $\G$. Then $\M_r(\G)\geq \M_r(\g)$, for $1\leq r\leq n-1$.
\end{corollary}
\begin{proof}
Again we may assume that the function $u$ in Theorem \ref{thm:comparison} is convex \cite[Lem. 1]{borbely2002}. If $M$ has constant  curvature $a$,  then $R_{ijk\ell}=
a(\delta_{ik}\delta_{j\ell}-\delta_{i\ell}\delta_{jk})$. Thus Theorem \ref{thm:comparison} yields
\be\label{eq:cc}
\M_r(\G)-\M_r(\g)=(r+1)\int_{\Omega\setminus D}\sigma_{r+1}(\kappa^u)\, -a(n-r)\int_{\Omega\setminus  D} 
\sigma_{r-1}(\kappa^u) .
\ee
By assumption $a\leq 0$, and since $u$ is convex, $\sigma_{r}(\kappa^u)\geq 0$, which completes the proof. 
\end{proof}

The above  result had been observed earlier by Solanes \cite[Cor. 9]{solanes2006}. It is due to the integral formula for quermassintegrals  \cite[Def. 2.1]{solanes2006}, which immediately yields that quermassintegrals of convex domains are increasing with respect to  inclusion. Monotonicity of total mean curvatures follows due to a formula \cite[Prop. 7]{solanes2006} relating quermassintegrals to total mean curvatures. As an application of the last corollary, one may extend definition of total mean curvatures to non-regular convex hypersurfaces as follows. If $\G$ is a convex hypersurface in a Cartan-Hadamard manifold, then its outer parallel hypersurface at distance $\epsilon$, denoted by $\G^\epsilon$, is $\C^{1,1}$ for all $\epsilon>0$ \cite[Lem. 2.6]{ghomi-spruck2022}. So $\M_r(\G^\epsilon)$ is well defined. By Corollary \ref{cor:parallel}, $\M_r(\G^\epsilon)$ is decreasing in $\epsilon$. Hence its limit as $\epsilon\to 0$ exists, and we may set
$
\M_r(\G):=\lim_{\epsilon\to 0}\M_r(\G^\epsilon).
$

Next we derive a formula which appears in Solanes \cite[(1) and (2)]{solanes2006}, and  follows from Gauss-Bonnet-Chern theorems \cite{chern1944,chern1945}; see also 
\cite[Cor. 8]{solanes2006}.
 Here $k!!$, when $k$ is a positive integer, stands for the product of all positive odd (even) integers up to $k$, when $k$ is odd (even). For $k\leq 0$ we set $k!!=1$.

\begin{corollary}\label{cor:solanes}
Let $\G$ be a closed $\C^{1,1}$ hypersurface in an $n$-manifold $M$ bounding a domain $\Omega$. Suppose that  $M$ has constant curvature $a$, and $\cl(\Omega)$ is diffeomorphic to a ball. Then
$$
\M_{n-1}(\G)
=
|\S^{n-1}|
-\sum_{i=1}^{\frac{n-(n\, \textup{mod}\, 2)}{2}} \frac{(2i-1)!!(n-2i-2)!!}{(n-2)!!} \,a^i  \M_{n-2i-1}(\G).
$$
\end{corollary}
\begin{proof}
Let $\phi\colon\cl(\Omega)\to B^n$ be a diffeomorphism to the unit ball in $\R^n$, and set $u(x):=|\phi(x)|^2$.
All regular level sets $\gamma$ of $u$ satisfy \eqref{eq:cc}. Furthermore, these level sets are convex near the minimum point $x_0$ of $u$, since $u$ has positive definite Hessian at $x_0$. So by Corollary \ref{cor:cc}, for these small level sets, 
$$
\M_r(S)\leq \M_r(\gamma)\leq \M_r(S'),
$$ 
where $S$ and $S'$ are geodesic spheres centered at $x_0$ such that $S$ is nested inside $\gamma$, and $\gamma$ is nested inside $S'$. Consequently, by Lemma \ref{lem:Mr0}, as $\g$ shrinks to $x_0$, $\M_{n-1}(\gamma)$ converges to $|\S^{n-1}|$, while $\M_r(\gamma)$ vanishes for $r\leq n-2$.  Thus, since $\sigma_n(\kappa^u)=0$, \eqref{eq:cc} yields
$$
\M_{n-1}(\G)=|\S^{n-1}|-a\int_{\Omega} 
\sigma_{n-2}(\kappa^u),
$$
and
$$
\int_{\Omega}\sigma_{r}(\kappa^u)\, 
=
\frac{1}{r}\M_{r-1}(\G)
+
\frac{a(n-r+1)}{r}\int_{\Omega} 
\sigma_{r-2}(\kappa^u) ,
$$
for $r\leq n-2$. Using these expressions iteratively completes the proof.
\end{proof}

Finally we include a characterization for hyperbolic balls, which extends to all mean curvatures a previous result of the authors on Gauss-Kronecker curvature \cite[Cor. 5.5]{ghomi-spruck2022}.

\begin{corollary}\label{cor:balls}
Let $M$ be a Cartan-Hadamard $n$-manifold with curvature $\leq a\leq 0$, and $B_\rho$ be a ball of radius $\rho$ in $M$. Then for $1\leq r\leq n-1$,
$$
\M_r(\partial B_\rho)\geq  \M_r(\partial B_\rho^a),
$$
where $B_\rho^a$ denotes a ball of radius $\rho$ in a manifold of constant curvature a.  Equality holds
 only if $B_\rho$ is isometric to $B_\rho^a$.
\end{corollary}
\begin{proof}
For $r=n-1$, the desired inequality has already been established \cite[Cor. 5.5]{ghomi-spruck2022}. Suppose then that $r\leq n-2$. 
We will show that
\be\label{eq:Gr-balls}
\M_r(\partial B_\rho)\geq (r+1)\int_{B_\rho}\sigma_{r+1}(\kappa^u)\, -a(n-r)\int_{B_\rho} 
\sigma_{r-1}(\kappa^u) \geq \M_r(\partial B_\rho^a).
\ee
Letting $u$ be the distance squared function from the center $o$ of $B_\rho$, and  $\g$ shrink to $o$ in Theorem \ref{thm:comparison}, yields the first inequality in \eqref{eq:Gr-balls} via Lemma \ref{lem:Mr0}.
The principal curvatures of $\partial B_{\rho}$ are
 bounded below by $\sqrt{-a}\coth(\sqrt{-a}\rho)$ \cite[p. 184]{karcher1989}, which are the principal curvatures of $\partial B_\rho^a$. Hence, the mean curvatures of $\partial B_\rho$ satisfy
 $$
 \sigma_{r}(\kappa^u)\geq \left(\begin{matrix}n-1\\r\end{matrix}\right)\big(\sqrt{-a}\coth (\sqrt{-a} \rho) \big)^{r}=\sigma^a_{r}(\kappa^u),
 $$
 where $\sigma^a_{r}(\kappa^u)$ are the mean curvatures of $\partial B^a_\rho$.
 Furthermore, if $A(\rho,\theta)d\theta$ denotes the volume element of $\partial B_\rho$ in geodesic spherical coordinates, then by \cite[(1.5.4)]{karcher1989}, 
$$
A(\rho,\theta)\geq \left(\frac{\sinh(\sqrt{-a}\rho)}{\sqrt{-a}}\right)^{n-1}=A^a(\rho,\theta),
$$
where $A^a(\rho,\theta)d\theta$ is the volume element of $\partial B_\rho^a$; see \cite[Cor. 5.5]{ghomi-spruck2022}. Thus,
\begin{eqnarray*}
\int_{B_\rho} \sigma_r(\kappa^u) 
\geq
\int_{0}^\rho \int_{\S^{n-1}} \sigma_{r}^a(\kappa^u) A^a(t,\theta)d\theta dt
=
\int_{B_\rho^a} \sigma_r^a(\kappa^u) ,
\end{eqnarray*}
which yields the second inequality in \eqref{eq:Gr-balls}. If
$\M_r(\partial B_\rho)=\M_r(\partial B_\rho^a)$,
then equality holds in the first inequality of \eqref{eq:Gr-balls}. So $K_{rn}=a$, i.e., the radial sectional curvatures of $B_\rho$ are constant, which forces $B_\rho$ to have constant curvature $a$ \cite[Lem. 5.4]{ghomi-spruck2022}. Hence $B_\rho$ is isometric to $B_\rho^a$.
\end{proof}

\addtocontents{toc}{\protect\setcounter{tocdepth}{0}}

\addtocontents{toc}{\protect\setcounter{tocdepth}{1}}
\bibliography{references}

\end{document}